\documentclass[12pt]{amsart}

\usepackage{amsmath}
\usepackage{amssymb}
\usepackage{pifont}
\usepackage{epsfig}

\usepackage{amscd}
\usepackage{latexsym}
\usepackage{amsfonts}

\usepackage{framed}
\usepackage{color}

\newtheorem{thm}{Theorem}[section]

\newtheorem{lem}[thm]{Lemma}

\newtheorem{rk}[thm]{Remark}

\numberwithin{equation}{section}

\begin{document}

\allowdisplaybreaks[4]

\title[Singular perturbation limit] 
{A singular perturbation limit of diffused interface energy with a fixed contact angle condition}

\author{Takashi Kagaya}
\address{Department of Mathematics, Tokyo Institute of Technology,
152-8551, Tokyo, Japan}
\email{kagaya.t.aa@m.titech.ac.jp}

\author{Yoshihiro Tonegawa}
\address{Department of Mathematics, Tokyo Institute of Technology,
152-8551, Tokyo, Japan}
\email{tonegawa@math.titech.ac.jp}

\medskip

\thanks{T. Kagaya is partially supported by JSPS Research Fellow Grant number 16J00547 and Y. Tonegawa is partially supported by JSPS KAKENHI Grant Numbers (A) 25247008 and (S) 26220702. } 

\begin{abstract}
We study a general asymptotic behavior of critical points of a diffused interface energy with a fixed contact angle condition defined on a domain $\Omega \subset \mathbb{R}^n$. 
We show that the limit varifold derived from the diffused energy 
satisfies a generalized contact angle condition on the boundary
under a set of assumptions.
\end{abstract}

\maketitle \setlength{\baselineskip}{18pt}

\section{Introduction}

In this paper, we consider a general asymptotic behavior of critical points of the energy functional 
\begin{equation}\label{energy} 
E_{\varepsilon}(u) = \int_{\Omega} \dfrac{\varepsilon|\nabla u|^2}{2}  + \dfrac{W(u)}{\varepsilon} \; dx + \int_{\partial \Omega} \sigma(u) \; d\mathcal{H}^{n-1} 
\end{equation}
under the restriction 
\begin{equation}\label{rest}
\int_\Omega u \; dx = m, 
\end{equation}
where $\varepsilon\in (0,1)$ is a small parameter, $\Omega \subset \mathbb{R}^n$ is a bounded domain, $u$ is a function defined on $\bar{\Omega}$, $W$ is a double well potential with strict minima at $\pm1$, $\sigma$ is a function on $\mathbb{R}$ and 
$m
\in(-|\Omega|,|\Omega|)$ is a fixed constant. $\mathcal H^{n-1}$ is the $n-1$-dimensional Hausdorff measure. 
According to the van der Waals-Cahn-Hilliard theory \cite{CH} and Cahn's approach \cite{C}, the energy \eqref{energy} is a typical energy modeling separation phenomena for capillary surfaces (see \cite{M2}). 
The function $u$, the strict minima of $W$ and the function $\sigma$ correspond to the normalized density of a multi-phase fluid, stable fluid phases and a contact energy density between the fluid and the container wall $\partial \Omega$, respectively.
The condition \eqref{rest} corresponds to fixing the total mass of the fluid in $\Omega$. 
If $E_\varepsilon(u_\varepsilon)$ is uniformly bounded with respect to $\varepsilon\in (0,1)$ for critical points $u_\varepsilon$ of $E_\varepsilon$, we may expect that the domain $\Omega$ is mostly divided into two regions $\{u_\varepsilon \approx 1\}$ and $\{u_\varepsilon \approx -1\}$ for sufficiently small $\varepsilon$. 

For energy minimizer of \eqref{energy}, Modica 
studied the contact angle condition in \cite{M2} 
within the framework of $\Gamma$-convergence. 
He showed the existence of energy minimizers $\{u_\varepsilon\}_{\varepsilon\in (0,1)}$ and the subsequential limit $u$ in $L^1$ as $\varepsilon \to 0$, and proved that $u = \pm 1$ a.e.\ on $\Omega$. 
Furthermore, in a weak sense, he showed under a suitable 
assumption on $\sigma$ that the contact angle $\theta$ formed by the boundary $\partial \Omega$ and the reduced boundary of $\{u=1\}$ in $\Omega$ 
is equal to
\begin{equation}\label{gam-cont-ang} 
\theta = \arccos\left(\dfrac{{\sigma}(1)-{\sigma}(-1)}{c_0}\right), 
\end{equation}
where 
\begin{equation}\label{c0} 
c_0 = \int^1_{-1} \sqrt{2W(s)} \; ds. 
\end{equation} 
The characterization of the contact angle condition 
is through
the energy minimality of the $\Gamma$-limit functional and it is essential that 
$u_{\varepsilon}$'s are global energy minimizers for the $\Gamma$-convergence 
argument. In view of the corresponding dynamical problem, however, it is interesting
to analyze the problem under a weaker assumption of being critical points. 
Our aim is to study the rigorous characterization of the contact angle condition due to the presence of the second term of \eqref{energy}
as $\varepsilon\rightarrow 0$. 

This line of research has been carried out by introducing a natural varifold associated with $u_\varepsilon$ (cf. \cite{HT,I,PT,RT}). 
Heuristically, the weight measure of the varifold behaves more or less like a surface measure of phase interface. 
One of the key tools to analyze a behavior of the varifold is the first variation.
In this paper, we focus on a behavior of the first variation of the associated varifolds up to the boundary and characterize the contact angle condition for the limit varifold along the line studied in \cite{KT}, as
described in Theorem \ref{thm:first-variation}. Roughly speaking, we
give a characterization of the tangential component of the 
first variation on $\partial\Omega$ which reduces to an appropriate
contact angle condition if all relevant quantities are smooth. 
Very closely related is the case of Neumann boundary condition,
namely, the case of $\sigma \equiv 0$. 
Mizuno and the second author \cite{MT} studied the gradient flow of \eqref{energy} in the case of $\sigma \equiv 0$  and analyzed a behavior of the first variation of the moving varifolds up to the boundary to derive a 
suitable Neumann boundary condition
for the limit Brakke flow. 

This paper is organized as follows. 
In Section \ref{eq-as}, we state known characterizations of limit varifold
in the interior of the domain due to \cite{HT,RT} along with setting our
notation. 
Section \ref{main} describes main results of the present paper,
which are the characterization of boundary behavior of the limit varifold. 
In Section \ref{proof}, we prove the main results and we give final remarks in Section \ref{remark}. 

\section{Preliminaries and interior behavior}\label{eq-as}

Let $\Omega \subset \mathbb{R}^n$ be a bounded domain
with smooth boundary $\partial\Omega$. We first describe the interior behavior of
general critical points of $E_{\varepsilon}$ under the following assumptions. 
Here we ignore the boundary conditions until the next section. 
\begin{itemize}
\item[(A1)]  $W\in C^\infty(\mathbb R)$ satisfies $W\geq 0$; $W(\pm1)=0$;
for some $\gamma\in (0,1)$, $W''(s)>0$ for all $|s|\geq \gamma$; $W$ has a unique local maximum in $(-1,1)$.
\item[(A2)] For a sequence $\{\varepsilon_i\}_{i=1}^\infty\subset (0,1)$ with $\lim_{i\rightarrow\infty}
\varepsilon_i=0$, 
$\{u_{\varepsilon_i}\}_{i=1}^\infty\subset C^{\infty}(\overline\Omega)$ 
satisfy
\begin{equation}\label{equation1}
-\varepsilon_i \Delta u_{\varepsilon_i} + \dfrac{W^{\prime}(u_{\varepsilon_i})}{\varepsilon_i} = \lambda_{\varepsilon_i}  \mbox{ on} \; \Omega
\end{equation}
for some $\lambda_{\varepsilon_i}\in\mathbb R$.  
\item[(A3)]
There exist constants $C>0$ and $E_0>0$ such that
\begin{equation}
\label{as-energy2}
\sup_i \|u_{\varepsilon_i}\|_{L^{\infty}(\Omega)}\leq C,\,\,\sup_i |\lambda_{\varepsilon_i}|\leq C
\end{equation}
and 
\begin{equation}
\label{as-energy}
\sup_i E_{\varepsilon_i}(u_{\varepsilon_i})\leq E_0.
\end{equation}
\end{itemize}
\begin{rk} 
Assumption (A1) says that $W$ is a
W-shaped function with two non-degenerate 
minima $\pm 1$. The equation \eqref{equation1} means that
$u_{\varepsilon_i}$ is a critical point of $E_{\varepsilon_i}$ with the volume 
constraint \eqref{rest}. Since we are primarily interested in $u_{\varepsilon}$ 
whose values are not far from $[-1,1]$ and whose energy remains $O(1)$, 
\eqref{as-energy2} and \eqref{as-energy} are reasonable assumptions. They
are the same set of assumptions in \cite{HT,RT}. 
\end{rk}
We next summarize
the direct consequences of (A1)-(A3) due to \cite{HT,RT} which give a fairly 
complete characterization of the limiting behavior in the interior of $\Omega$.
We introduce notation and definitions
related to varifolds to describe the results. 
We refer to \cite{A, S} for more information on varifold. 

Let $\mathbf{G}(n, n-1)$ denote the space of $(n-1)$-dimensional subspaces of $\mathbb{R}^n$. 
We also regard $S \in \mathbf{G}(n,n-1)$ as the orthogonal projection of $\mathbb{R}^n$ onto $S$, and write $S_1 \cdot S_2 = {\rm trace}(S_1 \circ S_2)$. 
For open $U\subset\mathbb R^n$, we say $V$ is an $(n-1)$-dimensional varifold in $U$ if $V$ is a Radon measure on $G_{n-1}(U) = U\times \mathbf{G}(n, n-1)$. 
Let $\mathbf{V}_{n-1}(U)$ denote the set of all $(n-1)$-dimensional varifolds. 
Convergence in the varifold sense means convergence in the usual sense of measures. 
For $V\in \mathbf{V}_{n-1}(U)$, we let $\|V\|$ be the weight measure of $V$.
Let ${\rm spt} \|V\|$ be the support of $\|V\|$.  
For $V \in \mathbf{V}_{n-1}(U)$, we define the first variation of $V$ by 
\[ \delta V(g) := \int_{G_{n-1}(U)} \nabla g(x) \cdot S \; dV(x,S) \]
for any vector field $g \in C^1_c ( U; \mathbb{R}^n)$. 
We also write the total variation of $\delta V$ by $\|\delta V\|$. 
If $\|\delta V\|$ is a Radon measure, we may apply the Radon-Nikodym theorem to $\delta V$ with respect to $\|V\|$. 
Writing the singular part of $\|\delta V\|$ with respect to $\|V\|$ as $\|\delta V\|_{\rm sing}$, we have $\|V\|$ measurable vector field $h$, 
$\|\delta V\|$ measurable $\nu_{\rm sing}$ with $|\nu_{\rm sing}|=1$ $\|\delta V\|$-a.e., and a Borel set $Z \subset U$ such that $\|V\|(Z)=0$ with,  
\begin{equation}\label{gene-mean} 
\delta V(g) = - \int_{U} \langle g, h\rangle \; d\|V\| + \int_Z \langle \nu_{\rm sing}, g \rangle \; d\|\delta V\|_{\rm sing} 
\end{equation} 
for all $g \in C^1_c(U; \mathbb{R}^n)$. 
We recall that $h$ is the generalized mean curvature vector of $V$, $\nu_{\rm sing}$ is the (outer-pointing) generalized co-normal of $V$ and 
$Z$ is the generalized boundary of $V$. If $V\in \mathbf{V}_{n-1}(U)$ satisfies
\begin{equation}\label{intvari}
V(\phi)=\int_{M} \phi(x,{\rm Tan}_x\, M)\Theta(x)\, d\mathcal H^{n-1}(x)
\end{equation}
for all $\phi\in C_c(G_{n-1}(U))$, where $M$ is an $\mathcal H^{n-1}$
measurable, countably $n-1$ rectifiable set, ${\rm Tan}_x\, M$ is the approximate
tangent space which exists for $\mathcal H^{n-1}$ a.e.~on $M$, $\Theta: M\rightarrow \mathbb N$ 
is an integer-valued $\mathcal H^{n-1}$ 
measurable function, $V$ is said to be integral. $\mathbf{IV}_{n-1}(U)$ 
denotes the set of all integral varifolds. Note that the $n-1$ dimensional density of $\|V\|$ (denoted by $\Theta(\|V\|,x)$)
exists $\|V\|$ a.e.~and is equal to $\Theta(x)$ in \eqref{intvari}.

Let $u_{\varepsilon_i}$ be the functions defined on $\overline{\Omega}$ satisfying (A1)-(A3). For each
 $u_{\varepsilon_i}$, we define a varifold $V_{\varepsilon_i}\in \mathbf{V}_{n-1}(\mathbb R^n)$ as follows. 
Define a Radon measure $\mu_{\varepsilon_i}$ on $\mathbb{R}^n$ by
\[ d\mu_{\varepsilon_i} := \frac{1}{c_0}\left(\dfrac{{\varepsilon_i}|\nabla u_{\varepsilon_i}|^2}{2} + \dfrac{W(u_{\varepsilon_i})}{{\varepsilon_i}}\right) d\mathcal{L}^n\lfloor_{\Omega}, \]
where $\mathcal{L}^n$ is the Lebesgue measure on $\mathbb{R}^n$ and 
$c_0$ is as in \eqref{c0}. 
Define $V_{\varepsilon_i} \in \mathbf{V}_{n-1}(\mathbb{R}^n)$ by 
\begin{equation*}\label{def-vari} 
V_{\varepsilon_i} (\phi) := \int_{\{|\nabla u_{\varepsilon_i} |\neq 0\}} \phi\left(x,I - \dfrac{\nabla u_{\varepsilon_i}}{|\nabla u_{\varepsilon_i}|} \otimes \dfrac{\nabla u_{\varepsilon_i}}{|\nabla u_{\varepsilon_i}|}\right) d\mu_{\varepsilon_i} 
\end{equation*}
for $\phi \in C_c(G_{n-1}(\mathbb{R}^n))$, where  $I$ is the $n\times n$ identity matrix. 
Then by the definition, we have
\begin{equation}\label{first-vari} 
\delta V_{\varepsilon_i}(g) = \int_{\{|\nabla u_{\varepsilon_i} |\neq 0\}} \nabla g \cdot \left(I - \dfrac{\nabla u_{\varepsilon_i}}{|\nabla u_{\varepsilon_i}|} \otimes \dfrac{\nabla u_{\varepsilon_i}}{|\nabla u_{\varepsilon_i}|}\right) \, d\mu_{\varepsilon_i} 
\end{equation}
for each $g \in C^1_c (\mathbb{R}^n ; \mathbb{R}^n)$. In addition, we define a function 
\[\xi_{\varepsilon_i}:=\frac{1}{c_0}\Big(\frac{\varepsilon_i |\nabla u_{\varepsilon_i}|^2}{2}
-\frac{W(u_{\varepsilon_i})}{\varepsilon_i}\Big)\]
on $\overline\Omega$ and $\xi_{\varepsilon_i}:=0$ on $\mathbb R^n\setminus \overline\Omega$. This is called a discrepancy in the literature. The following two
theorems are
direct consequences of \cite{HT,RT}. 
\begin{thm}\label{mt1}
(\cite[Theorem 1]{HT}) Under the assumptions (A1)-(A3),
there exists a subsequence (denoted by the same index) such that 
\[ \lambda_{\varepsilon_i}\rightarrow\lambda,\,\, u_{\varepsilon_i}\rightarrow
u\mbox{ in }L^1(\Omega),\,\, u\in BV(\Omega),\,\,\,\, V_{\varepsilon_i}\rightarrow V\,\,\mbox{in the varifold sense of $\mathbf{V}_{n-1}(\mathbb R^n)$}, \]\[ |\xi_{\varepsilon_i}|\,d\mathcal L^n
\rightarrow d\xi\,\,\mbox{in the sense of Radon measures on $\mathbb R^n$}.\]
Moreover, 
\begin{itemize}
\item[(1)] $u(x)=\pm 1$ for $\mathcal L^n$ a.e. on $\Omega$,
\item[(2)] $V\lfloor_{G_{n-1}(\Omega)} \in \mathbf{IV}_{n-1}(\Omega)$,
\item[(3)] ${\rm spt}\,\xi\subset \partial \Omega$ and $\xi\leq \|V\|\lfloor_{
\partial\Omega}$,  
\item[(4)] $\Omega\cap {\rm spt}\,\partial^*\{u=1\}\subset {\rm spt}\,\|V\|$ and
$u_{\varepsilon_i}\rightarrow\pm 1$ locally uniformly on $\Omega\setminus
{\rm spt}\,\|V\|$.
\end{itemize}
\end{thm}
By the well-known property of $BV$ functions (see for example \cite{EG}), 
away from the reduced boundary 
\[M:=\Omega\cap \partial^*\{u=1\}\]
of $\{u=1\}$ in $\Omega$, we may define
$u(x)\in \{\pm 1\}$ for $\mathcal H^{n-1}$ a.e.~$x\in \Omega\setminus M$.  
We also write $\nabla u/|\nabla u|$ which exists for $\mathcal H^{n-1}$ a.e.~on
$M$ as the inward-pointing unit normal to $\partial^*\{u=1\}$. 
\begin{thm}\label{mt2}
(\cite[Theorem 3.2]{RT}) Let $\lambda, u, V, M$ be as above. Then we have
the following. 
\begin{itemize}
\item[(a)] $V\lfloor_{G_{n-1}(\Omega)}$ (as an element of $\mathbf{V}_{n-1}(\Omega)$)
has a generalized mean curvature $h$ with $\|\delta V\|_{\rm sing}=0$ in $\Omega$.  
We have $\mathcal H^{n-1}(M\setminus {\rm spt}\,\|V\|)=0$. 
\item[(b)] $V$ has a locally constant mean curvature in $\Omega$, namely, 
\[h=\left\{\begin{array}{ll} \frac{2\lambda}{c_0}\frac{\nabla u}{|\nabla u|} & \mathcal H^{n-1}\, a.e.\mbox{ on } M,\\
0 & \mathcal H^{n-1} a.e.\mbox{ on } {\rm spt}\,\|V\|\cap\Omega\setminus
M\end{array}\right.\]
and
\[\Theta(\|V\|,x)=\left\{\begin{array}{ll} \mbox{odd} & \mathcal H^{n-1}\, a.e.\mbox{ on } M, \\
\mbox{even} & \mathcal H^{n-1} \, a.e.\mbox{ on }{\rm spt}\,\|V\|\cap \Omega \setminus 
M. \end{array} \right.\]
\item[(c)] If $\lambda\neq 0$, then ``odd'' in (b) is replaced by ``1''. 
\item[(d)] If $\lambda>0$, then $\mathcal H^{n-1}(\{u=1\}\cap{\rm spt}\,\|V\|
\cap \Omega\setminus M)=0$. If $\lambda<0$, then $\mathcal H^{n-1}(\{u=-1\}\cap{\rm spt}\,\|V\|
\cap \Omega\setminus M)=0$.
\end{itemize}
\end{thm}
The portion of ``even multiplicity part'' ${\rm spt}\,\|V\|\cap \Omega\setminus M$ may be
regarded as a hidden boundary, in the sense that it does not appear as a boundary
of $\{u=1\}$. 
Just to clarify the point of above claim, consider the case when $\lambda=0$. Then
(b) says that $V$ is stationary in $\Omega$ with the density parity as described. 
If $\lambda>0$, then the even multiplicity part which has 0 mean curvature
only appears (if it does exist non-trivially) in the region of $\{u=-1\}$ due to (d). 
In the following, Theorem \ref{mt2} is not used and it is presented for the convenience
of the reader. 
\begin{rk}
It is important to note for the following section that \cite[Theorem 1]{HT}
proves $|\xi_{\varepsilon_i}|\rightarrow 0$ on $\Omega$. This leaves the possibility
of having non-trivial measure $\xi$ living only on $\partial \Omega$. 
When $\Omega$ is strictly convex and $\sigma=0$, 
it is proved that $\xi=0$ in \cite{MT}. We conjecture that $\xi=0$ also for non-trivial 
$\sigma$ and under some geometric condition (such as convexity) on $\Omega$. Due to 
the trivial inequality $\xi\leq \|V\|$, if $\|V\|\lfloor_{
\partial\Omega}=0$, then we have $\xi=0$. Thus, if the measures $\mu_{\varepsilon_i}$ 
do not concentrate on $\partial\Omega$, we have $\xi=0$ in particular. 
\label{disrk}
\end{rk}
\section{boundary behavior}\label{main}
In addition to (A1)-(A3) in the previous section, we now consider the following
three assumptions. 
\begin{itemize}
\item[(A4)] A given function $\sigma\in C^{\infty}(\mathbb R)$ satisfies 
\begin{equation}\label{as-func}
|\sigma^{\prime}(s)| \le C_1\sqrt{2 W(s)}
\end{equation} 
for some $C_1\in [0,1)$ and for all $s\in\mathbb R$. 
\item[(A5)] The functions $\{u_{\varepsilon_i}\}$ as in (A2) satisfy
\begin{equation}\label{equation2}
\varepsilon_i \langle \nabla u_{\varepsilon_i}, \nu \rangle = - \sigma^{\prime}(u_{\varepsilon_i}) \mbox{ on} \; \partial \Omega, 
\end{equation}
where $\nu$ is the outer unit normal vector field on $\partial \Omega$.
\item[(A6)] $\xi=0$, where $\xi$ is as in Theorem \ref{mt1} (3). 
\end{itemize}
From a heuristic argument as well as the $\Gamma$-convergence result of \cite{M2}, note that we expect the energy $E_{\varepsilon}$
should behave like
\[E_{\varepsilon}(u_{\varepsilon})\approx c_0 \mathcal H^{n-1}(\Omega\cap
\partial\{u=1\})+(\sigma(1)-\sigma(-1))\mathcal H^{n-1}(\partial\Omega\cap \{u=1\})+{\rm Constant}.\]
Imposing (A4) ensures that $|\sigma(1)-\sigma(-1)|\leq \int_{-1}^1
|\sigma'(s)|\, ds\leq C_1\int_{-1}^1\sqrt{2W(s)}\, ds<c_0$. Physically, 
this ensures that the contact energy density $|\sigma(1)-\sigma(-1)|$ of the interface 
$\{u_{\varepsilon}\approx 1\}$ with $\partial \Omega$ 
is strictly smaller than the surface tension
density $c_0$ of the interface inside of $\Omega$. As $|\sigma(1)-\sigma(-1)|\nearrow
c_0$, we expect to have a ``perfect wetting'' (see \cite{C}) of the interface. The equality
\eqref{equation2} is satisfied for critical points of \eqref{energy} with the volume
constraint \eqref{rest}, as one can check easily by taking the first variation of
$E_{\varepsilon_i}$. For (A6), as mentioned in Remark \ref{disrk}, we do not know
in general that this is satisfied under the assumptions (A1)-(A5). 
However, it is a reasonable assumption since we expect
$\|V\|\lfloor_{\partial\Omega}=0$ (and thus $\xi\leq \|V\|\lfloor_{\partial\Omega}=0$)
unless the situation is somewhat pathological. We also note that 
adding the stability assumption (that is, the second variation of $E_{\varepsilon}$ is non-negative)
does not appear helpful to show $\xi=0$ on $\partial \Omega$, despite the result of 
$\Gamma$-convergence of \cite{M2}.    

In the following, we first describe the behavior of $u_{\varepsilon_i}\lfloor_{\partial
\Omega}$. 
\begin{thm}\label{behavior-u}
Under the assumptions (A1)-(A5) (thus leaving out (A6)), there exist a subsequence
(denoted by the same index) and a function $\tilde{u} \in BV(\partial \Omega)$ 
such that
\begin{align*}
&u_{\varepsilon_i}\lfloor_{\partial \Omega} \to \tilde{u} \; \; \mathcal{H}^{n-1}\,
a.e.\mbox{ on } \partial \Omega, 
\\
&\tilde{u} = \pm1 \; \; \mathcal{H}^{n-1}\, a.e.\mbox{ on} \; \partial \Omega, 
\end{align*}
where $u_{\varepsilon_i} \lfloor_{\partial \Omega}$ is the restriction of $u_{\varepsilon_i}$ to $\partial \Omega$. 
\end{thm}

In general, the trace of $u$ (obtained in Theorem \ref{mt1}) on
$\partial \Omega$ may not coincide with $\tilde u$, as one can construct a sequence of 
critical points of $E_{\varepsilon}$ with $\sigma=0$ which converge to $u=1$ on 
$\Omega$ while $u_{\varepsilon}\lfloor_{\partial\Omega}\approx -1$ (see \cite[Section 8]{MT}).
The next result is the main theorem of the paper. 
\begin{thm}\label{thm:first-variation}
Under the assumptions (A1)-(A6), let $V$ be as in Theorem \ref{mt1} and let $\tilde u$
be as in Theorem \ref{behavior-u}.  Then we have the following. Let $\theta$ be defined
as in \eqref{gam-cont-ang}.
\begin{itemize}
\item[(1)] 
The total variation $\|\delta V\|(\mathbb R^n)=\|\delta V\|(\overline\Omega)$ (as an element of $\mathbf{V}_{n-1}(\mathbb R^n)$)
is finite. 
\item[(2)]
For any vector field $g \in C(\partial\Omega; \mathbb{R}^n)$ such that $\langle g, \nu\rangle = 0$ on $\partial \Omega$, we have
\begin{equation}\label{tan-first} 
\delta V\lfloor_{\partial \Omega}(g) = \cos\theta\int_{\partial^*\{x\in \partial
\Omega \,:\, \tilde u(x)=1\}} \langle g,\tau\rangle \; d\mathcal{H}^{n-2}, 
\end{equation}
where $\tau(x)\in {\rm Tan}_x(\partial \Omega)$ is the $\mathcal H^{n-2}$ measurable 
unit inward-pointing normal to
$\partial^*\{x\in \partial\Omega \, :\, \tilde u(x)=1\}$ which exists $\mathcal H^{n-2}$ a.e.~
on $\partial^*\{x\in \partial\Omega \, :\, \tilde u(x)=1\}$. 
\end{itemize}
\end{thm}
The equality \eqref{tan-first} gives a complete description of the tangential component of
the first variation on the boundary. Also, \eqref{tan-first} may be considered as a generalized 
contact angle condition satisfied for a pair of varifold $V$ and $\tilde u$. To see this, consider
a case that $\|V\|=\mathcal H^{n-1}\lfloor_M$ and $M$ is a smooth hypersurface having a
smooth boundary $\partial M\subset \partial \Omega$.
Then the first variation $\delta V\lfloor_{\partial\Omega}(g)$ is represented as 
\[\int_{\partial M} \langle g, \tilde \nu\rangle\, d\mathcal H^{n-2},\]
where $\tilde \nu$ is the unit outward-pointing co-normal to $\partial M$. 
Then \eqref{tan-first} shows that $\partial M\cap \{\tilde \nu\neq \nu\}=\partial^*\{\tilde u=1\}$ and the angle formed by $\tilde \nu$ and $\tau$ is $\theta$. Away from $\partial^*
\{\tilde u=1\}$, $\partial M$ (if such set is non-empty) intersects with $\partial\Omega$ orthogonally. Hence, more
precisely, we should say that the contact angle condition with angle $\theta$ is satisfied 
on $\partial^*\{\tilde u=1\}$. For further remark on the implication of \eqref{tan-first}, see
Section \ref{remark}.

\section{Proof of Theorem \ref{behavior-u} and \ref{thm:first-variation}}\label{proof}

Throughout this section, we will replace the notation $\varepsilon_i$ by $\varepsilon$. First, we derive a formula for the first variation $\delta V_\varepsilon$. 
\begin{lem}
For $u_{\varepsilon}$ satisfying \eqref{equation1} and \eqref{equation2} 
and for $g \in C^1_c(\mathbb{R}^n; \mathbb{R}^n)$, we have 
\begin{equation}\label{first-vari2}
\begin{aligned}
c_0 \delta V_\varepsilon (g) &=\; \int_{\Omega \cap \{|\nabla u_\varepsilon| \neq 0\}} \nabla g \cdot \dfrac{\nabla u_\varepsilon}{|\nabla u_\varepsilon|} \otimes \dfrac{\nabla u_\varepsilon}{|\nabla u_\varepsilon|} \left(\dfrac{\varepsilon|\nabla u_\varepsilon|^2}{2} - \dfrac{W(u_\varepsilon)}{\varepsilon}\right) \, dx \\ 
&\; - \int_{\Omega \cap \{|\nabla u_\varepsilon |= 0\}} \nabla g \cdot I \dfrac{W(u_\varepsilon)}{\varepsilon} \; dx + \int_\Omega \lambda_\varepsilon u_\varepsilon \, {\rm div} g \; dx \\
&\; + \int_{\partial \Omega} \left(\dfrac{\varepsilon |\nabla u_\varepsilon|^2}{2} + \dfrac{W(u_\varepsilon)}{\varepsilon} - \lambda_\varepsilon \right) \langle g, \nu\rangle \; d\mathcal{H}^{n-1} + \int_{\partial \Omega} \sigma^{\prime}(u_\varepsilon) \langle \nabla u_\varepsilon, g \rangle \; d\mathcal{H}^{n-1} \\
&=:\; I^\varepsilon_1(g) + I^\varepsilon_2(g) + I^\varepsilon_3(g) + I^\varepsilon_4(g) + I^\varepsilon_5(g). 
\end{aligned}
\end{equation}
\end{lem}
\begin{proof}
We fix a vector field $g \in C^1_c(\mathbb{R}^n; \mathbb{R}^n)$ and calculate the right-hand side of \eqref{first-vari}. 
Using the boundary condition \eqref{equation2} and by integration by parts, we have 
\begin{equation}\label{part1}
\begin{aligned}
&\; \int_{\Omega\cap\{|\nabla u_\varepsilon |\neq 0\}} \nabla g \cdot I \dfrac{\varepsilon|\nabla u_\varepsilon|^2}{2} \; dx = \int_\Omega \nabla g \cdot I \dfrac{\varepsilon|\nabla u_\varepsilon|^2}{2} \; dx \\
=& \; \int_{\partial \Omega} \dfrac{\varepsilon |\nabla u_\varepsilon|^2}{2} \langle g, \nu \rangle \; d\mathcal{H}^{n-1} - \varepsilon \int_\Omega \nabla^2 u_\varepsilon \cdot \nabla u_\varepsilon \otimes g \; dx \\
=& \; \int_{\partial \Omega} \dfrac{\varepsilon |\nabla u_\varepsilon|^2}{2} \langle g, \nu \rangle \; d\mathcal{H}^{n-1} + \varepsilon \int_\Omega \nabla g \cdot \nabla u_\varepsilon \otimes \nabla u_\varepsilon - \langle \nabla u_\varepsilon, \nabla\langle \nabla u_\varepsilon, g \rangle\rangle \; dx \\
=& \; \int_{\partial \Omega} \dfrac{\varepsilon |\nabla u_\varepsilon|^2}{2} \langle g, \nu \rangle + \sigma^{\prime}(u_\varepsilon) \langle \nabla u_\varepsilon, g \rangle \; d\mathcal{H}^{n-1} + \varepsilon \int_\Omega \Delta u_\varepsilon \langle \nabla u_\varepsilon, g \rangle + \nabla g \cdot \nabla u_\varepsilon \otimes \nabla u_\varepsilon \; dx. 
\end{aligned}
\end{equation}
Also by integration by parts, we obtain 
\begin{equation}\label{part2}
\begin{aligned}
&\; \int_{\Omega\cap\{|\nabla u_\varepsilon| \neq 0\}} \dfrac{W(u_\varepsilon)}{\varepsilon} \nabla g \cdot I \; dx = \int_\Omega \dfrac{W(u_\varepsilon)}{\varepsilon} \nabla g \cdot I \; dx - \int_{\Omega \cap \{|\nabla u_\varepsilon |= 0\}} \dfrac{W(u_\varepsilon)}{\varepsilon} \nabla g \cdot I \; dx\\
= &\; \int_{\partial \Omega} \dfrac{W(u_\varepsilon)}{\varepsilon} \langle g, \nu\rangle \; d\mathcal{H}^{n-1} - \int_\Omega \dfrac{W^{\prime}(u_\varepsilon)}{\varepsilon} \langle \nabla u_\varepsilon, g \rangle \; dx - \int_{\Omega \cap \{|\nabla u_\varepsilon| = 0\}} \dfrac{W(u_\varepsilon)}{\varepsilon} \nabla g \cdot I \; dx. 
\end{aligned}
\end{equation}
Substituting \eqref{part1} and \eqref{part2} into \eqref{first-vari}, we have by the interior equation \eqref{equation1}
\begin{align*} 
c_0\delta V_\varepsilon (g) =&\; \int_{\Omega \cap \{|\nabla u_\varepsilon |\neq 0\}} \nabla g \cdot \dfrac{\nabla u_\varepsilon}{|\nabla u_\varepsilon|} \otimes \dfrac{\nabla u_\varepsilon}{|\nabla u_\varepsilon|} \left(\dfrac{\varepsilon|\nabla u_\varepsilon|^2}{2} - \dfrac{W(u_\varepsilon)}{\varepsilon}\right) \, dx \\
&\; - \int_{\Omega \cap \{|\nabla u_\varepsilon |= 0\}} \nabla g \cdot I \dfrac{W(u_\varepsilon)}{\varepsilon} \; dx -\int_\Omega \lambda_\varepsilon \langle \nabla u_\varepsilon, g \rangle \; dx \\
&\; + \int_{\partial \Omega} \left(\dfrac{\varepsilon |\nabla u_\varepsilon|^2}{2} + \dfrac{W(u_\varepsilon)}{\varepsilon} \right) \langle g, \nu\rangle \; d\mathcal{H}^{n-1} + \int_{\partial \Omega} \sigma^{\prime}(u_\varepsilon) \langle \nabla u_\varepsilon, g \rangle \; d\mathcal{H}^{n-1}. 
\end{align*}
By integration by parts for the third term of right-hand side, we obtain \eqref{first-vari2}.
\end{proof}
\begin{lem}\label{lem:bo-ener}
Under the assumption of (A1)-(A5), 
there exists a constant $C_2>0$ depending only on $\Omega, C, E_0, C_1$ such that 
\begin{equation}\label{bo-ener} 
\int_{\partial \Omega} \dfrac{\varepsilon |\nabla u_\varepsilon|^2}{2} + \dfrac{W(u_\varepsilon)}{\varepsilon} \; d\mathcal H^{n-1} \le C_2. 
\end{equation}
\end{lem}
\begin{proof}
We choose a smooth function $f : \overline\Omega\rightarrow \mathbb R$ which satisfies
$\nabla f=\nu$ on $\partial\Omega$. For example, $f(x)=-{\rm dist}\,(x,\partial \Omega)$ 
near $\partial\Omega$ with a suitable truncation away from $\partial\Omega$ suffices. We then use $g=\nabla f$ in \eqref{first-vari2}. 
By the definition \eqref{first-vari} and \eqref{as-energy}, 
we have $c_0 |\delta V_{\varepsilon}
(\nabla f)|\leq E_0 \sup\|f\|_{C^2}$ so the left-hand side of 
\eqref{first-vari2} is bounded depending only on $E_0$ and $\Omega$.
The terms $I_1^{\varepsilon}(\nabla f), I_2^{\varepsilon}(\nabla f)$ 
and $I_3^{\varepsilon}(\nabla f)$ are also bounded by a constant 
depending only on $C, E_0,\Omega$. Thus we have
\[\int_{\partial\Omega} 
\Big(\frac{\varepsilon |\nabla u_{\varepsilon}|^2}{2}+
\frac{W(u_{\varepsilon})}{\varepsilon}\Big)\, d\mathcal H^{n-1}
\leq -\int_{\partial\Omega} \sigma'(u_{\varepsilon})\langle\nabla
u_{\varepsilon},\nu\rangle\, d\mathcal H^{n-1}+c(C,E_0,\Omega)\]
where $\nabla f\lfloor_{\partial\Omega}=\nu$ is used. 
By Young's inequality and the assumption \eqref{as-func},
\begin{equation}\label{eq5}
\begin{aligned}
\left| \int_{\partial \Omega}\sigma^{\prime}(u_{\varepsilon})  \langle \nabla u_{\varepsilon}, \nu \rangle \; d\mathcal{H}^{n-1} \right| \le& \; \int_{\partial \Omega} \dfrac{\varepsilon C_1 |\nabla u_\varepsilon|^2}{2} + \dfrac{(\sigma^{\prime}(u_\varepsilon))^2}{2C_1\varepsilon} \; d\mathcal{H}^{n-1} \\
\le& \; C_1 \int_{\partial \Omega} \dfrac{\varepsilon |\nabla u_\varepsilon|^2}{2} + \dfrac{W(u_\varepsilon)}{\varepsilon} \; d\mathcal{H}^{n-1}. 
\end{aligned}
\end{equation}
Since $C_1\in [0,1)$, we have the conclusion by setting $C_2=c(C,E_0,\Omega)/(1-C_1)$.
\end{proof}

\begin{proof}[Proof of Theorem \ref{behavior-u}]
Once we have \eqref{bo-ener}, a well-known argument (\cite{M1,St})
leads to the conclusion. For the
convenience of the reader, we include the argument. Let 
\[\Phi(s):=\int_{-1}^s\sqrt{2W(s)}\, ds\]
and define a function $w_{\varepsilon}:=\Phi\circ u_{\varepsilon}$ on $\partial \Omega$.
Since $|\nabla w_{\varepsilon}|=|\nabla u_{\varepsilon}|\sqrt{2W(u_{\varepsilon})}
\leq \frac{\varepsilon|\nabla  u_{\varepsilon}|^2}{2}+\frac{W(u_{\varepsilon})}{\varepsilon}$,
we have a uniform bound on $\|\nabla w_{\varepsilon}\|_{L^1(\partial \Omega)}$. 
Since $|\nabla_{\partial \Omega} w_\varepsilon| \le |\nabla w_\varepsilon|$, with $L^{\infty}$
bounds of \eqref{as-energy2}, the well-known compactness theorem of BV functions applies.
Thus we have a subsequence and $\tilde w\in BV(\partial\Omega)$
such that $w_{\varepsilon}\rightarrow \tilde w$ pointwise for a.e.~on $\partial\Omega$. 
Since $\Phi^{-1}$ is continuous, $u_{\varepsilon}$ converges a.e.~pointwise to $\tilde u :=
\Phi^{-1}\circ \tilde w$. Fatou's lemma with $\int_{\partial\Omega}W(u_{\varepsilon})\,
d\mathcal H^{n-1}\rightarrow 0$ also proves that $\tilde u=\pm 1$. This ends the proof.
\end{proof}
\begin{proof}[Proof of Theorem \ref{thm:first-variation} (1)]
Fixing $g\in C_c^1(\mathbb R^n;\mathbb R^n)$, we have
$\lim_{i\rightarrow\infty}\delta V_{\varepsilon_i}(g)=\delta V(g)$
due to the varifold convergence. In \eqref{first-vari2}, due to 
(A6), we have $\lim_{i\rightarrow\infty} |I_1^{\varepsilon_i}(g)|+
|I_2^{\varepsilon_i}(g)|=0$. By Theorem \ref{mt1}, we have
\begin{equation}\label{i-3}
\lim_{i\rightarrow\infty} I_3^{\varepsilon_i}(g)=\lambda\int_{\Omega} u\,{\rm div}\,g\,dx=-2\lambda \int_{M}\Big\langle g,\frac{\nabla u}{|
\nabla u|}\Big\rangle\, d\mathcal H^{n-1}+\lambda\int_{\partial
\Omega} u\,\langle g,\nu\rangle\, d\mathcal H^{n-1},
\end{equation}
where $M=\Omega\cap \partial^*\{u=1\}$. Using \eqref{bo-ener} and a similar
argument as in \eqref{eq5}, we can show $|I_4^{\varepsilon_i}(g)|+
|I_5^{\varepsilon_i}(g)|\leq c\sup|g|$, where $c$ is independent of $g$
or $i$. 
Combined all these estimates, we show that $|\delta V(g)|\leq 
c\sup |g|$ and $\|\delta V\|(\overline\Omega)$ is finite.
\end{proof}
\begin{proof}[Proof of Theorem \ref{thm:first-variation} (2)]
It suffices to prove the claim for $g\in C_c^1(\mathbb R^n ; \mathbb R^n)$ with $\langle g,\nu\rangle=0$ on $\partial\Omega$, since
the general $C_c(\mathbb R^n;\mathbb R^n)$ case can be proved by 
approximation. For such $g$, in \eqref{i-3}, the last term vanishes
and also $I_4^{\varepsilon}(g)=0$ in \eqref{first-vari2}. 
For $I_5^{\varepsilon}(g)$, we have $\langle \nabla u_{\varepsilon},g\rangle=\langle \nabla_{\partial\Omega} u_{\varepsilon},g\rangle$ due to $\langle g,\nu\rangle=0$. 
Thus, by the divergence theorem on $\partial \Omega$, we have
\[ 
I^\varepsilon_5(g) =  
\int_{\partial \Omega} \sigma^{\prime}(u_\varepsilon) \langle \nabla_{\partial \Omega} u_\varepsilon, g \rangle \; d\mathcal{H}^{n-1} = -\int_{\partial \Omega} \sigma(u_\varepsilon) \,{\rm div}_{\partial \Omega}\, g \; d\mathcal{H}^{n-1}. 
\]
These lead to the conclusion that 
\begin{equation}
\label{fi1}
c_0\delta V(g)=-2\lambda\int_M \Big\langle g,\frac{\nabla u}{|
\nabla u|}\Big\rangle\, d\mathcal H^{n-1}-\int_{\partial\Omega}
\sigma(\tilde u)\, {\rm div}_{\partial\Omega}\, g\, d\mathcal H^{n-1}.
\end{equation}
Since $\tilde u\in BV(\partial\Omega)$ with values in $\{\pm 1\}$, $\partial^*\{
\tilde u=1\}$ and the inward-pointing unit normal $\tau$ are
well-defined, and
\begin{equation}
\label{fi2}
-\int_{\partial\Omega}\sigma(\tilde u)\, {\rm div}_{\partial\Omega}
\, g\, d\mathcal H^{n-1}=(\sigma(1)-\sigma(-1))\int_{\partial^*
\{\tilde u=1\}} \langle \tau,g\rangle\, d\mathcal H^{n-2}.
\end{equation}
Since we are interested in obtaining $\delta V\lfloor_{\partial\Omega}$,
and since $M\subset \Omega$, we obtain \eqref{tan-first} from 
\eqref{fi1} and \eqref{fi2}. 
\end{proof}

\section{Additional remarks}\label{remark}
\subsection{The case $\|V\|(\partial\Omega)=0$}
If we further assume that $\|V\|(\partial\Omega)=0$, then, 
non-trivial $\delta V\lfloor_{\partial\Omega}$ is necessarily
singular with respect to $\|V\|\lfloor_{\partial \Omega}$. Thus
using the notation of \eqref{gene-mean}, we conclude from
\eqref{tan-first} that
\[\int_Z \langle\nu_{\rm sing}, g\rangle\, d\|\delta V\|_{\rm sing}
=\cos\theta\int_{\partial^*\{\tilde u=1\}} \langle g,\tau\rangle\,
d\mathcal H^{n-2}\]
for $g\in C(\partial \Omega,\mathbb R^n)$ with $\langle g,\nu\rangle=0$ on $\partial\Omega$. 
If $Z=\partial^*\{\tilde u=1\}$ and $\|\delta V\|_{\rm sing}\lfloor_Z
=\mathcal H^{n-2}\lfloor_Z$, then we have a clear-cut statement 
that $\nu_{\rm sing}-\langle\nu_{\rm sing},\nu\rangle\nu=(\cos\theta)
\,\tau$ on $Z$, which says that the generalized 
co-normal of $V$ satisfies the contact angle condition with 
angle $\theta$. Unfortunately, even in this case, we can only conclude
that $\partial^*\{\tilde u=1\}\subset Z$. Also we do not know in 
general if $\|\delta V\|_{\rm sing}\lfloor_{\partial^*\{\tilde u=1\}}
=\mathcal H^{n-2}\lfloor_{\partial^*\{\tilde u=1\}}$. On the other 
hand, on $Z\setminus \partial^*\{\tilde u=1\}$, even though we equally
do not know what $\|\delta V\|_{\rm sing}$ is in general, we may 
conclude $\nu_{\rm sing}=\nu$, $\|\delta V\|_{\rm sing}$ a.e.~ since
the right-hand side is 0 away from $\partial^*\{\tilde u=1\}$. Thus
the right-angle condition is simpler to describe than other
non-right-angle conditions. 
\subsection{The case $\|V\|(\partial\Omega)>0$}
It may be somewhat counter-intuitive to imagine that 
the measures $\mu_{\varepsilon}$ may ``pile-up'' on the boundary as
 $\varepsilon\rightarrow 0$, resulting in $\|V\|(\partial\Omega)>0$.
For $\sigma=0$ and $\Omega=B_1(0)$, it is not difficult to 
construct such example, however, as described in \cite[Section 8]{MT}
(see also \cite{MW, MNW} for examples for more general domains and
of higher-multiplicity concentration). Interestingly, even if $\|V\|(\partial\Omega)>0$, as long as $\xi=0$, results in the paper 
still hold true. We expect that the presence of non-trivial $\|V\|$ in
$\partial\Omega$ affects the normal component of the first variation, 
but not the tangential one. In all known examples where boundary concentration
of $\|V\|$ occurs, $\xi$ is zero. 
\subsection{Monotonicity formula}
In \cite{KT}, motivated by the present paper, we introduce a notion
of generalized contact angle condition for varifold and derive a 
monotonicity
formula valid up to the boundary. The condition in \cite{KT}
is even weaker than the one obtained in Theorem 
\ref{thm:first-variation} in that we do not need to have a bounded first variation 
up to the boundary. Thus the result of \cite{KT} applies to $V$ 
in this paper and up to the boundary monotonicity formula can be
obtained. For $\sigma=0$ and convex $\Omega$, in \cite{T}, the
similar up to the boundary monotonicity formula was obtained even
for the diffused energy (i.e.~before letting $\varepsilon\rightarrow 0$). 
To gain a better understanding on $V$ obtained in this paper, 
it is desirable to establish such
monotonicity formula for diffused energy since one can conclude 
a better convergence of interface to ${\rm spt}\,\|V\|$. This is ultimately
connected to getting a good estimate on the discrepancy up to the
boundary and showing $\xi=0$, along the line of logics in \cite{HT,I,MT}.

\end{document}